\documentclass[preprint, times, numbers,3p, sort&compress]{elsarticle}
\usepackage{amssymb}
\usepackage{amsthm}
\usepackage{graphicx}%
\usepackage{multirow}%
\usepackage{amsmath,amssymb,amsfonts}%
\usepackage{mathrsfs}%
\usepackage{dsfont}%
\usepackage[title]{appendix}%
\usepackage{xcolor}%
\usepackage{textcomp}%
\usepackage{manyfoot}%
\usepackage{booktabs}%
\usepackage{algorithm}%
\usepackage{algorithmicx}%
\usepackage{algpseudocode}%
\usepackage{listings}%
\usepackage{url}
\usepackage{makecell}
\usepackage{hyperref}
\usepackage{pgfplots}
\usepackage{tikz}
\usepackage{subcaption}
\usepackage{multicol}

\pgfplotsset{compat=1.18}

\journal{Indagationes Mathematicae}

\newtheorem{theorem}{Theorem}[section]
\newtheorem{proposition}[theorem]{Proposition}

\newtheorem{corollary}[theorem]{Corollary}

\newtheorem{remark}[theorem]{Remark}

\newcommand{\R}{\mathbb R}
\newcommand{\C}{\mathbb C}
\newcommand{\Z}{\mathbb Z}
\newcommand{\CP}{\mathbb{CP}}
\newcommand{\dd}{\,\mathrm d}
\newcommand{\sn}{\operatorname{sn}}

\newcommand{\ns}{\operatorname{ns}}
\newcommand{\Pn}{\mathcal P_N}

\definecolor{loopzero}{HTML}{D62728}
\definecolor{loopone}{HTML}{1F77B4}
\definecolor{loopa}{HTML}{2CA02C}
\definecolor{loopinf}{HTML}{9467BD}
\begin{document}

\begin{frontmatter}

\title{Some Results on Lam\'e--Heun Spectral Geometry}

\author[aff1]{U.S. Idiong\corref{cor1}}
\ead{idiongus@afued.edu.ng}
\cortext[cor1]{Corresponding author.}
\address[aff1]{Department of Mathematics, Adeyemi Federal University of Education, Ondo, Nigeria}

\begin{abstract}
We develop a unified spectral and complex-analytic theory of the Lamé equation through its canonical Heun realization. Starting with ellipsoidal separation, we derive essential elements, including the Lamé-preserving anharmonic parameter action and finite accessory spectral polynomials. We construct the global Heun monodromy representation on the four-punctured sphere, connected to the $(2,2,2,2)$ pillowcase orbifold and the $SL_2(\C)$ character surface. The accessory parameter is framed within a Riemann–Hilbert map that links spectral values to monodromy conditions. Utilizing Seifert–van Kampen theory reveals the fundamental group of the punctured sphere, and the Nielsen–Schreier rank formula applies to monodromy coverings. The vanishing of higher homotopy groups for unbranched covers and the emergence of a Riemann–Hurwitz genus formula post-compactification are noted. Colored Schreier graphs are employed for a visual calculus of the covering data, connecting elliptic geometry and topology.
\end{abstract}


\end{frontmatter}

\section{Introduction}
The Lam\'e equation is a classical second-order differential equation arising from separation of Laplace's equation in ellipsoidal coordinates. It also appears in finite-gap spectral theory, elliptic Schr\"odinger operators, special-function theory, and the study of Fuchsian equations \cite{WhittakerWatson,Ince,Arscott,Erdelyi,Maier2004}.

Several equivalent forms occur in the literature. First, we consider the relationship between the equivalent forms in the order stated below.
\begin{equation}
\begin{aligned}
\text{Laplace equation}
&\longrightarrow
\text{ellipsoidal Lam\'e equation}
\longrightarrow
\text{algebraic Lam\'e equation}\\
&\longrightarrow
\text{Weierstrass form}
\longleftrightarrow
\text{Jacobi form}
\longrightarrow
\text{Heun equation}.
\end{aligned}
\label{eq:chain}
\end{equation}
The first aim of this paper is to present this chain without the common normalization errors that arise from ellipsoidal derivatives, the point at infinity, or Jacobi imaginary translations. The second aim is to exploit the Heun embedding rather than stop at it. Once the Lam\'e equation is written as a special Heun equation, the anharmonic subgroup of the $192$-element Heun group acts explicitly on the modulus and accessory parameter. This produces a sixfold parameter-equivalence orbit, special harmonic and equianharmonic strata, and affine maps between spectral problems. For even integral coupling, the same Heun realization gives a finite polynomial sector and a tridiagonal accessory matrix.

Throughout, ``equivalent'' is used carefully. A change of variable gives differential-equation equivalence; a gauge factor gives local-solution equivalence; and a spectral equivalence requires the corresponding eigenvalue normalization to be tracked explicitly.

The global extension developed here is equally careful about the word ``monodromy.'' For fixed exponent data, the accessory parameter does not determine a single universal finite group. Rather, it determines a representation of the fundamental group of the four-punctured sphere into $GL_2(\C)$, whose image may be generic irreducible and infinite, reducible, projectively unitary, or finite in exceptional cases. The paper therefore derives the complete local conjugacy data, the global presentation, the Lam\'e orbifold quotient, and the induced covering-space invariants instead of claiming an impossible finite enumeration of all accessory-dependent images. This viewpoint follows the classical Heun monodromy and finite-gap literature while making its relation to the Lam\'e accessory spectrum explicit \cite{DLMF,Takemura2008,Xia2021}.

The paper is organized into four principal sections. Section~2 collects the classical analytic,
elliptic, spectral, monodromy, and covering-space results used later. Section~3 contains the
results derived in this work, including the Lam\'e-preserving Heun symmetry action, accessory
spectral polynomials, coefficient asymptotics, monodromy alignment, and the resulting global
spectral--topological consequences. Section~4 discusses the results, their physical applications,
and open problems suggested by the analysis.

\section{Mathematical Preliminaries}
\label{sec:preliminaries}

This section gathers the established results and standard constructions used throughout the
paper. The purpose is to separate the classical analytic and topological background from the
new results developed in Section~\ref{sec:main-results}.

\subsection{Elliptic-function background}
Let
\begin{equation}
\Lambda=\{2m\omega_1+2n\omega_2:m,n\in\Z\},
\qquad \omega_2/\omega_1\notin\R.
\end{equation}
The Weierstrass function is
\begin{equation}
\wp(z)=\frac1{z^2}+\sum_{\omega\in\Lambda\setminus\{0\}}
\left[\frac1{(z-\omega)^2}-\frac1{\omega^2}\right],
\end{equation}
and satisfies
\begin{equation}
(\wp'(z))^2=4\wp(z)^3-g_2\wp(z)-g_3
=4\prod_{j=1}^{3}(\wp(z)-e_j),
\label{eq:wp-cubic}
\end{equation}
where $e_j=\wp(\omega_j)$, $\omega_3=\omega_1+\omega_2$, and $e_1+e_2+e_3=0$. The discriminant is
\begin{equation}
\Delta=g_2^3-27g_3^2
=16(e_1-e_2)^2(e_1-e_3)^2(e_2-e_3)^2.
\end{equation}
We assume $\Delta\neq0$ whenever elliptic uniformization is used.

The function $\wp$ is doubly periodic, whereas the associated sigma and zeta functions are quasi-periodic. In particular,
\begin{equation}
\sigma(z+2\omega_j)=-\sigma(z)e^{2\eta_j(z+\omega_j)},
\qquad
\zeta(z+2\omega_j)=\zeta(z)+2\eta_j.
\end{equation}
This distinction matters when solution formulas are transported between elliptic representations.

\subsection{Confocal ellipsoidal coordinates and Lam\'e separation}
Let $a>b>c>0$. The confocal quadrics are
\begin{equation}
\frac{x^2}{a^2+s}+\frac{y^2}{b^2+s}+\frac{z^2}{c^2+s}=1,
\label{eq:confocal}
\end{equation}
with roots $s=\lambda,\mu,\nu$ at a generic point. Define
\begin{equation}
\phi(s)=(s+a^2)(s+b^2)(s+c^2).
\end{equation}
Then
\begin{align}
x^2&=\frac{(a^2+\lambda)(a^2+\mu)(a^2+\nu)}{(a^2-b^2)(a^2-c^2)},\label{eq:xell}\\
y^2&=\frac{(b^2+\lambda)(b^2+\mu)(b^2+\nu)}{(b^2-a^2)(b^2-c^2)},\\
z^2&=\frac{(c^2+\lambda)(c^2+\mu)(c^2+\nu)}{(c^2-a^2)(c^2-b^2)}.
\end{align}
Taking logarithmic derivatives of \eqref{eq:xell} gives, for example,
\begin{equation}
\frac{\partial x}{\partial\lambda}=\frac{x}{2(a^2+\lambda)},
\qquad
\frac{\partial x}{\partial\mu}=\frac{x}{2(a^2+\mu)},
\qquad
\frac{\partial x}{\partial\nu}=\frac{x}{2(a^2+\nu)}.
\end{equation}
The corresponding formulas for $y$ and $z$ follow cyclically.

\begin{proposition}
The coordinates $(\lambda,\mu,\nu)$ are mutually orthogonal wherever the three roots are distinct. Their Lam\'e scale factors are
\begin{equation}
h_\lambda^2=\frac{(\lambda-\mu)(\lambda-\nu)}{4\phi(\lambda)},
\qquad
h_\mu^2=\frac{(\mu-\lambda)(\mu-\nu)}{4\phi(\mu)},
\qquad
h_\nu^2=\frac{(\nu-\lambda)(\nu-\mu)}{4\phi(\nu)}.
\end{equation}
\end{proposition}

\begin{proof}
The normal to the quadric $s=\mathrm{const.}$ is proportional to
\begin{equation}
\left(\frac{x}{a^2+s},\frac{y}{b^2+s},\frac{z}{c^2+s}\right).
\end{equation}
For distinct roots $s_1,s_2$, subtraction of the two equations \eqref{eq:confocal} after division by $s_1-s_2$ gives zero scalar product of the corresponding normals. The scale factors follow by substituting the coordinate formulas into $h_s^2=|\partial_s(x,y,z)|^2$.
\end{proof}

The Euclidean Laplacian therefore becomes
\begin{align}
\Delta={}&
\frac{4\sqrt{\phi(\lambda)}}{(\lambda-\mu)(\lambda-\nu)}
\frac{\partial}{\partial\lambda}
\left(\sqrt{\phi(\lambda)}\frac{\partial}{\partial\lambda}\right)
\nonumber\\
&+
\frac{4\sqrt{\phi(\mu)}}{(\mu-\lambda)(\mu-\nu)}
\frac{\partial}{\partial\mu}
\left(\sqrt{\phi(\mu)}\frac{\partial}{\partial\mu}\right)
\nonumber\\
&+
\frac{4\sqrt{\phi(\nu)}}{(\nu-\lambda)(\nu-\mu)}
\frac{\partial}{\partial\nu}
\left(\sqrt{\phi(\nu)}\frac{\partial}{\partial\nu}\right).
\label{eq:ellipsoid-laplacian}
\end{align}
For $\Psi=L(\lambda)M(\mu)N(\nu)$, separation of $\Delta\Psi=0$ gives a common ordinary differential equation
\begin{equation}
4\phi(s)X''(s)+2\phi'(s)X'(s)-[n(n+1)s+C]X(s)=0.
\label{eq:ellipsoid-lame}
\end{equation}
Equivalently,
\begin{equation}
X''+\frac12\left(\frac1{s+a^2}+\frac1{s+b^2}+\frac1{s+c^2}\right)X'
-\frac{n(n+1)s+C}{4\phi(s)}X=0.
\end{equation}
Its regular singular points on the Riemann sphere are $-a^2,-b^2,-c^2,\infty$. The finite exponent pairs are $(0,1/2)$, and the exponents at infinity are $-n/2$ and $(n+1)/2$.

\subsection{Algebraic and Weierstrass forms}
Set
\begin{equation}
c_0=\frac{a^2+b^2+c^2}{3},
\qquad
p=-s-c_0,
\end{equation}
and define
\begin{equation}
e_1=a^2-c_0,
\qquad e_2=b^2-c_0,
\qquad e_3=c^2-c_0.
\end{equation}
Then $e_1+e_2+e_3=0$. With
\begin{equation}
B=n(n+1)c_0-C,
\end{equation}
Eq.~\eqref{eq:ellipsoid-lame} becomes
\begin{equation}
X''(p)+\frac12\sum_{j=1}^{3}\frac1{p-e_j}X'(p)
-\frac{n(n+1)p+B}{4(p-e_1)(p-e_2)(p-e_3)}X(p)=0.
\label{eq:algebraic-lame}
\end{equation}
Let $P(p)=\prod_{j=1}^{3}(p-e_j)$. Then
\begin{equation}
4P(p)X''(p)+2P'(p)X'(p)-[n(n+1)p+B]X(p)=0.
\end{equation}

\begin{theorem}[Weierstrass uniformization]
Under $p=\wp(u)$, Eq.~\eqref{eq:algebraic-lame} is equivalent to
\begin{equation}
X''(u)-[n(n+1)\wp(u)+B]X(u)=0,
\label{eq:wp-lame}
\end{equation}
or
\begin{equation}
\left[-\frac{\dd^2}{\dd u^2}+n(n+1)\wp(u)\right]X=-BX.
\end{equation}
\end{theorem}

\begin{proof}
Using \eqref{eq:wp-cubic}, $(\wp')^2=4P(\wp)$. Moreover,
\begin{equation}
\frac{\dd^2X}{\dd u^2}=(\wp')^2X''(\wp)+\wp''X'(\wp).
\end{equation}
Since $2\wp''=4P'(\wp)$, substitution into the algebraic equation cancels the first-derivative term and gives \eqref{eq:wp-lame}.
\end{proof}

\subsection{Jacobi representation and the classical Heun embedding}
Assume $e_1>e_2>e_3$ and set
\begin{equation}
A=\sqrt{e_1-e_3},
\qquad
m=k^2=\frac{e_2-e_3}{e_1-e_3}.
\end{equation}
The standard identity is
\begin{equation}
\wp(u)=e_3+(e_1-e_3)\ns^2(Au\mid m).
\label{eq:wp-ns}
\end{equation}
With $x=Au$, Eq.~\eqref{eq:wp-lame} becomes
\begin{equation}
X''(x)-\left[n(n+1)\ns^2(x\mid m)+E\right]X(x)=0,
\qquad
E=\frac{n(n+1)e_3+B}{e_1-e_3}.
\label{eq:ns-lame}
\end{equation}
The Jacobi imaginary-translation identity
\begin{equation}
\ns(x\mid m)=k\,\sn(x+iK'\mid m)
\end{equation}
requires the translated variable $v=x+iK'$. Hence
\begin{equation}
-X''(v)+n(n+1)k^2\sn^2(v\mid m)X=-EX.
\label{eq:jacobi-lame}
\end{equation}
No identification of $\sn(x+iK')$ with $\sn x$ is made.

Now consider the standard Jacobi Lam\'e equation
\begin{equation}
Y''(u)+[h-\nu(\nu+1)k^2\sn^2(u\mid m)]Y(u)=0.
\label{eq:jacobi-standard}
\end{equation}
Set $z=\sn^2(u\mid m)$. Then
\begin{equation}
\left(\frac{\dd z}{\dd u}\right)^2=4z(1-z)(1-k^2z),
\qquad
\frac{\dd^2z}{\dd u^2}=2[1-2(1+k^2)z+3k^2z^2].
\end{equation}
Equation \eqref{eq:jacobi-standard} becomes the canonical Heun equation
\begin{equation}
Y''+\left(\frac{\gamma}{z}+\frac{\delta}{z-1}+\frac{\epsilon}{z-a_{\mathrm H}}\right)Y'
+\frac{\alpha\beta z-q}{z(z-1)(z-a_{\mathrm H})}Y=0,
\label{eq:heun}
\end{equation}
with
\begin{equation}
a_{\mathrm H}=k^{-2},
\qquad
\gamma=\delta=\epsilon=\frac12,
\end{equation}
and
\begin{equation}
\alpha=-\frac{\nu}{2},
\qquad
\beta=\frac{\nu+1}{2},
\qquad
q=-\frac{h}{4k^2}.
\label{eq:lame-heun-parameters}
\end{equation}
The Fuchs relation $\alpha+\beta+1=\gamma+\delta+\epsilon$ is immediate.

\subsection{Classical Heun transformation group}
\label{sec:classical-heun-group}

The full transformation group of the general Heun equation has order $192$ and is isomorphic to the Coxeter group $W(D_4)$ \cite{Maier2007}. The full group does not preserve the Lam\'e parameter locus
\begin{equation}
\mathcal L=\left\{(a,q,\alpha,\beta,\tfrac12,\tfrac12,\tfrac12)\right\},
\end{equation}
because exponent flips or exchanges with infinity generally change the finite exponent differences. The subgroup generated by permutations of the three finite singularities does preserve $\mathcal L$. It is the anharmonic group, isomorphic to $S_3$.
The full group-theoretic description above is classical; it is used only as background for the
Lam\'e-preserving subgroup derived in Section~\ref{sec:symmetry}.

\subsection{Trigonometric degeneration}
Let $q_{\rm nome}=e^{i\pi\tau}$ with $\operatorname{Im}\tau\to+\infty$. For $2\omega_1=1$,
\begin{equation}
\wp(z)=\pi^2\csc^2(\pi z)-\frac{\pi^2}{3}+O(q_{\rm nome}^2).
\end{equation}
Hence
\begin{align}
-\frac{\dd^2}{\dd z^2}+n(n+1)\wp(z)
\longrightarrow{}&
-\frac{\dd^2}{\dd z^2}+n(n+1)\pi^2\csc^2(\pi z)
\nonumber\\
&-\frac{\pi^2}{3}n(n+1).
\end{align}
The constant term is a spectral shift. In Heun language this degeneration corresponds to a singular limit in which two regular singularities effectively coalesce after rescaling; consequently the full $192$-element Fuchsian symmetry should not be carried unchanged into the confluent problem.

\subsection{Classical completeness, local series, and convergence geometry}
\label{sec:classical-spectral}

The Heun embedding makes it possible to connect three questions that are usually treated separately \cite{DLMF,Eastham,MagnusWinkler}: Hilbert-space completeness of Lam\'e eigenfunctions, analytic continuation of the corresponding Fuchs--Frobenius solutions, and asymptotic growth of their Taylor coefficients. The first is a global operator-theoretic property; the second and third are local complex-analytic properties. Their intersection is precisely where the accessory parameter becomes spectral.

\subsubsection{Completeness on a period cell}

Let $0<m<1$ and consider the Jacobi Lam\'e operator
\begin{equation}
H_{\rm L}=-\frac{\dd^2}{\dd u^2}+\nu(\nu+1)m\sn^2(u\mid m)
\label{eq:lame-selfadjoint}
\end{equation}
on the interval $[0,K(m)]$. Since $m\sn^2(u\mid m)$ is real and bounded on this compact interval, every separated regular boundary condition defines a standard Sturm--Liouville problem.

\begin{theorem}[Finite-interval completeness]
\label{thm:finite-complete}
Let $\nu\in\R$, $0<m<1$, and let $H_{\rm L}$ be equipped with any of the four separated self-adjoint endpoint conditions
\begin{align}
&Y'(0)=Y'(K)=0, 
&&Y(0)=Y'(K)=0,\nonumber\\
&Y'(0)=Y(K)=0,
&&Y(0)=Y(K)=0.
\label{eq:lame-bcs}
\end{align}
Then $H_{\rm L}$ is self-adjoint with compact resolvent in $L^2(0,K)$. Its spectrum is real, discrete and unbounded above, and its eigenfunctions form a complete orthonormal basis of $L^2(0,K)$.
\end{theorem}

\begin{proof}
The multiplication operator by $\nu(\nu+1)m\sn^2(u\mid m)$ is bounded and self-adjoint on $L^2(0,K)$. The second derivative with any of the separated conditions in \eqref{eq:lame-bcs} is self-adjoint with compact resolvent. Bounded self-adjoint perturbation preserves self-adjointness and compactness of the resolvent. The spectral theorem for self-adjoint operators with compact resolvent therefore yields a complete orthonormal eigenbasis.
\end{proof}

The four endpoint conditions correspond to the classical parity classes of Lam\'e functions. Under
\begin{equation}
z=\sn^2(u\mid m),
\label{eq:zmap-complete}
\end{equation}
one has $z\sim u^2$ near $u=0$ and
\begin{equation}
1-z\sim (1-m)(u-K)^2
\end{equation}
near $u=K$. Hence even and odd endpoint parity corresponds to choosing local Heun exponents $0$ and $1/2$, respectively, at $z=0$ and $z=1$.

On the whole line the operator is periodic rather than compact-resolvent. Its appropriate completeness statement is the Floquet--Bloch direct-integral decomposition. Thus the finite-interval Lam\'e functions provide ordinary Hilbert-space completeness, whereas the whole-line periodic problem provides generalized Bloch completeness. The finite polynomial sector of Section~\ref{sec:qes} is complete only inside its invariant polynomial module and must not be confused with either of these global notions.

\subsubsection{Local Heun series and exact coefficient recurrence}

Return to the canonical Heun equation
\begin{equation}
y''+
\left(
\frac{\gamma}{z}+\frac{\delta}{z-1}+\frac{\epsilon}{z-a}
\right)y'
+
\frac{\alpha\beta z-q}{z(z-1)(z-a)}y=0,
\label{eq:heun-series-eq}
\end{equation}
with the Fuchs relation
\begin{equation}
\alpha+\beta+1=\gamma+\delta+\epsilon.
\end{equation}
For the exponent-zero solution at $z=0$, write
\begin{equation}
y(z;q)=\sum_{n=0}^{\infty}c_n(q)z^n,
\qquad c_0=1.
\label{eq:heun-series}
\end{equation}
The coefficients satisfy
\begin{equation}
a\gamma c_1-qc_0=0
\end{equation}
and, for $n\geq1$,
\begin{equation}
R_n c_{n+1}-(Q_n+q)c_n+P_n c_{n-1}=0,
\label{eq:heun-recurrence-general}
\end{equation}
where
\begin{align}
P_n&=(n-1+\alpha)(n-1+\beta),\nonumber\\
Q_n&=n\left((n-1+\gamma)(1+a)+a\delta+\epsilon\right),\nonumber\\
R_n&=a(n+1)(n+\gamma).
\label{eq:PQR}
\end{align}
For the Lam\'e locus,
\begin{equation}
a=\frac1m,
\qquad
\gamma=\delta=\epsilon=\frac12,
\qquad
\alpha=-\frac{\nu}{2},
\qquad
\beta=\frac{\nu+1}{2},
\qquad
q=-\frac{h}{4m},
\label{eq:lame-heun-asymp-params}
\end{equation}
and the recurrence reduces to
\begin{align}
&a(n+1)\left(n+\frac12\right)c_{n+1}
-\bigl((1+a)n^2+q\bigr)c_n
\nonumber\\
&\hspace{30mm}
+\left(n-1-\frac{\nu}{2}\right)
\left(n+\frac{\nu-1}{2}\right)c_{n-1}=0.
\label{eq:lame-coeff-recurrence}
\end{align}
In particular,
\begin{equation}
c_1=\frac{2q}{a}=-\frac h2.
\end{equation}
Thus the accessory parameter enters the Taylor sequence from its first nontrivial coefficient onward.

\subsubsection{Radius of convergence and singularity geometry}

For an ordinary expansion point $z_0\notin\{0,1,a\}$, analytic ODE theory gives convergence at least up to the nearest finite singularity. Hence the generic Taylor radius is
\begin{equation}
R(z_0)=\min\{|z_0|,|z_0-1|,|z_0-a|\},
\label{eq:generic-radius}
\end{equation}
unless the nearest singularity is removable for the particular solution. At $z=0$ the natural Frobenius disk is
\begin{equation}
R_0=\min\{1,|a|\}.
\end{equation}
For the physical Lam\'e range $0<m<1$, one has $a=1/m>1$ and therefore
\begin{equation}
R_0=1.
\label{eq:R0}
\end{equation}
Similarly,
\begin{equation}
R_1=\min\left\{1,\frac{1-m}{m}\right\},
\qquad
R_a=\frac{1-m}{m}.
\label{eq:other-radii}
\end{equation}
The change at $m=1/2$ is geometric: it is exactly the point where the distances from $z=1$ to $0$ and to $a=1/m$ coincide.

\subsection{General analytic continuation and Heun monodromy}
\label{sec:classical-monodromy}

Let
\begin{equation}
X_a=\mathbb{CP}^1\setminus\{0,1,a,\infty\},
\qquad a\notin\{0,1\}.
\end{equation}
For fixed Heun parameters, let $\mathscr L_q$ denote the rank-two local system of solutions of the canonical Heun equation on $X_a$. Analytic continuation along a based loop $\gamma$ acts linearly on a chosen fundamental solution basis and defines the monodromy representation
\begin{equation}
\rho_q:\pi_1(X_a,z_*)\longrightarrow GL_2(\mathbb C).
\label{eq:monodromy-representation}
\end{equation}
Its image
\begin{equation}
\mathcal M(q)=\rho_q\bigl(\pi_1(X_a,z_*)\bigr)
\end{equation}
is the linear monodromy group, and its image in $PGL_2(\mathbb C)$ is the projective monodromy group $\mathbb P\mathcal M(q)$.

Choose positively oriented loops $\gamma_0,\gamma_1,\gamma_a,\gamma_\infty$ about the four punctures with the standard relation
\begin{equation}
\gamma_0\gamma_1\gamma_a\gamma_\infty=1.
\label{eq:loop-relation}
\end{equation}
Writing
\begin{equation}
M_s=\rho_q(\gamma_s),
\qquad s\in\{0,1,a,\infty\},
\end{equation}
we obtain
\begin{equation}
M_0M_1M_aM_\infty=I.
\label{eq:monodromy-product}
\end{equation}
The local exponent data of the Heun equation imply the spectral conjugacy classes
\begin{align}
\operatorname{spec}(M_0)&=\{1,e^{-2\pi i\gamma}\},\\
\operatorname{spec}(M_1)&=\{1,e^{-2\pi i\delta}\},\\
\operatorname{spec}(M_a)&=\{1,e^{-2\pi i\epsilon}\},\\
\operatorname{spec}(M_\infty)&=\{e^{2\pi i\alpha},e^{2\pi i\beta}\},
\end{align}
where the last line uses the local coordinate $t=1/z$ at infinity. The Fuchs relation guarantees compatibility of the determinants with \eqref{eq:monodromy-product}. These local eigenvalues are independent of the accessory parameter $q$; the global connection matrices and hence the conjugacy class of the full representation vary with $q$ \cite{DLMF,Xia2021}.

\subsubsection{Local monodromy types}

When an exponent difference is nonintegral, the corresponding local monodromy is diagonalizable. When an exponent difference is integral, logarithmic Frobenius terms may occur and the local monodromy can acquire a Jordan block. Thus the local possibilities are semisimple nonresonant, scalar or apparent, and resonant logarithmic. The global image $\mathcal M(q)$ is substantially richer: for generic parameters it is irreducible and infinite; at exceptional accessory values it may become reducible or projectively unitary; and if the projective image is finite, the classical classification of finite subgroups of $PGL_2(\mathbb C)$ restricts it to cyclic, dihedral, tetrahedral $A_4$, octahedral $S_4$, or icosahedral $A_5$ type. The statement is a classification of possible finite projective images, not an assertion that every type occurs for every Lam\'e parameter choice.

\begin{proposition}[Accessory dependence of global monodromy]
Fix $a$ and all exponent parameters. Then the local conjugacy classes of $M_0,M_1,M_a,M_\infty$ are fixed, whereas the simultaneous conjugacy class of the tuple
\begin{equation}
(M_0,M_1,M_a,M_\infty)
\end{equation}
varies with the accessory parameter $q$. Consequently, the accessory problem is naturally a Riemann--Hilbert problem: determine $q$ from prescribed global monodromy data.
\end{proposition}

This is the monodromy-theoretic meaning of the accessory parameter. Isomonodromic sets of Heun accessory parameters have been studied through Painlev\'e VI and tau-function methods \cite{Xia2021}.

\subsection{Covering-space and homotopy background}
\label{sec:coverings}

\subsubsection{The fundamental group of the Heun domain}

The punctured sphere $X_a$ is homotopy equivalent to a wedge of three circles. This can be derived directly by the Seifert--van Kampen theorem; the covering and higher-homotopy statements below use standard surface and covering-space theory \cite{Hatcher,Forster}.

Choose an open set $U$ containing the three finite punctures and a base point, with small deleted disks connected to the base point by disjoint corridors, and let $V$ be a punctured neighborhood of infinity enlarged so that $U\cap V$ is path connected. The loops about the four punctures generate the respective local fundamental groups, while the overlap identifies the boundary word. Van Kampen yields
\begin{equation}
\pi_1(X_a,z_*)
=\left\langle
\gamma_0,\gamma_1,\gamma_a,\gamma_\infty
\;\middle|\;
\gamma_0\gamma_1\gamma_a\gamma_\infty=1
\right\rangle
\cong F_3.
\label{eq:van-kampen-presentation}
\end{equation}
In particular, one may eliminate $\gamma_\infty$ and take $\gamma_0,\gamma_1,\gamma_a$ as free generators.

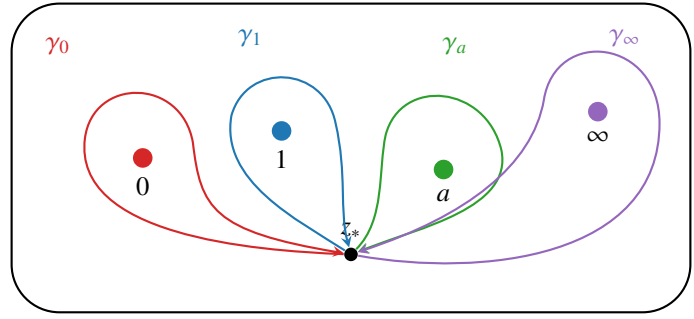
\begin{figure}[t]
\centering
\begin{tikzpicture}[scale=1.02,>=stealth]
  \draw[rounded corners=18pt,thick] (-4.4,-2.0) rectangle (4.4,2.0);
  \node[fill=loopzero,circle,inner sep=2.6pt,label=below:$0$] (p0) at (-2.7,0) {};
  \node[fill=loopone,circle,inner sep=2.6pt,label=below:$1$] (p1) at (-0.9,0.35) {};
  \node[fill=loopa,circle,inner sep=2.6pt,label=below:$a$] (pa) at (1.2,-0.15) {};
  \node[fill=loopinf,circle,inner sep=2.6pt,label=below:$\infty$] (pi) at (3.2,0.6) {};
  \node[black,circle,fill=black,inner sep=1.8pt,label=above:$z_*$] (b) at (0,-1.25) {};
  \draw[loopzero,thick,->] (b) .. controls (-1.3,-1.2) and (-3.6,-1.0) .. (-3.45,0.25)
       .. controls (-3.35,1.05) and (-2.2,1.05) .. (-2.05,0.2)
       .. controls (-1.95,-0.65) and (-1.8,-0.95) .. (b);
  \draw[loopone,thick,->] (b) .. controls (-0.7,-0.8) and (-1.7,-0.3) .. (-1.55,0.55)
       .. controls (-1.35,1.25) and (-0.35,1.2) .. (-0.2,0.45)
       .. controls (-0.05,-0.2) and (-0.15,-0.7) .. (b);
  \draw[loopa,thick,->] (b) .. controls (0.45,-0.8) and (0.35,-0.15) .. (0.45,0.3)
       .. controls (0.65,1.05) and (1.75,0.95) .. (1.95,0.1)
       .. controls (2.1,-0.65) and (0.9,-0.9) .. (b);
  \draw[loopinf,thick,->] (b) .. controls (1.7,-1.55) and (4.0,-1.3) .. (4.0,0.45)
       .. controls (4.0,1.55) and (2.65,1.7) .. (2.5,0.75)
       .. controls (2.35,-0.1) and (1.5,-0.75) .. (b);
  \node[loopzero] at (-3.8,1.45) {$\gamma_0$};
  \node[loopone] at (-1.3,1.55) {$\gamma_1$};
  \node[loopa] at (1.35,1.45) {$\gamma_a$};
  \node[loopinf] at (3.55,1.55) {$\gamma_\infty$};
\end{tikzpicture}
\caption{Chromatic van Kampen generators on the four-punctured Heun sphere. The colors are retained in subsequent covering graphs: red for $0$, blue for $1$, green for $a$, and purple for $\infty$.}
\label{fig:chromatic-loops}
\end{figure}

\subsubsection{The monodromy covering}

The kernel
\begin{equation}
K_q=\ker\rho_q
\end{equation}
is a normal subgroup of $\pi_1(X_a)$. Let
\begin{equation}
p_q:\widetilde X_q\longrightarrow X_a
\end{equation}
be the connected covering corresponding to $K_q$. By construction, analytic continuation of the Heun solution basis along every closed loop in $\widetilde X_q$ is trivial. Hence the pulled-back local system is globally single valued.

\begin{theorem}[Topology of the monodromy cover]
\label{thm:monodromy-cover}
The monodromy cover satisfies
\begin{equation}
\pi_1(\widetilde X_q)\cong\ker\rho_q.
\end{equation}
If $\mathcal M(q)$ is finite, then $p_q$ is a regular finite-sheeted cover of degree $|\mathcal M(q)|$ with deck group isomorphic to $\mathcal M(q)$. If $\mathcal M(q)$ is infinite, the same construction gives an infinite-sheeted regular cover.

In every case,
\begin{equation}
\pi_n(\widetilde X_q)=0,
\qquad n\ge2.
\label{eq:higher-monodromy-cover}
\end{equation}
\end{theorem}

\begin{proof}
The first statement is the covering-subgroup correspondence. Since $X_a$ is homotopy equivalent to a graph, its universal cover is a tree and hence contractible. Every connected cover of $X_a$ has the same universal cover. Therefore every connected monodromy cover is aspherical, which proves \eqref{eq:higher-monodromy-cover}.
\end{proof}

This answers the higher-homotopy problem for unbranched Heun monodromy coverings completely: all higher homotopy groups vanish.

\subsubsection{Finite covers and the Nielsen--Schreier rank}

Let $p:\widetilde X\to X_a$ be a connected unbranched covering of degree $d<\infty$. Since $\pi_1(X_a)\cong F_3$, the corresponding subgroup has index $d$. The Nielsen--Schreier formula gives
\begin{equation}
\operatorname{rank}\pi_1(\widetilde X)
=1+d(3-1)
=2d+1.
\label{eq:nielsen-schreier}
\end{equation}
Thus every connected $d$-sheeted unbranched Heun cover has free fundamental group $F_{2d+1}$, independently of the detailed permutation cycle structure.

\begin{corollary}
For a finite monodromy group $G=\mathcal M(q)$, the kernel monodromy cover has
\begin{equation}
\pi_1(\widetilde X_q)\cong F_{2|G|+1}
\end{equation}
provided the linear monodromy representation itself defines the finite regular deck quotient $G$.
\end{corollary}

\subsubsection{Permutation monodromy and compactified branched covers}

A finite-sheeted topological cover is encoded by a transitive permutation representation
\begin{equation}
\tau:\pi_1(X_a)\longrightarrow S_d.
\end{equation}
Set
\begin{equation}
\sigma_0=\tau(\gamma_0),\quad
\sigma_1=\tau(\gamma_1),\quad
\sigma_a=\tau(\gamma_a),\quad
\sigma_\infty=\tau(\gamma_\infty),
\end{equation}
so that
\begin{equation}
\sigma_0\sigma_1\sigma_a\sigma_\infty=1.
\end{equation}
Let $c(\sigma)$ denote the number of cycles, including fixed points, in a permutation. Filling in the punctures gives a compact branched covering
\begin{equation}
\overline p:\overline X\longrightarrow\mathbb{CP}^1.
\end{equation}
Riemann--Hurwitz yields
\begin{equation}
2-2g
=2d-
\sum_{s\in\{0,1,a,\infty\}}
\bigl(d-c(\sigma_s)\bigr),
\end{equation}
and therefore
\begin{equation}
g
=1+d-\frac12
\sum_{s\in\{0,1,a,\infty\}}c(\sigma_s).
\label{eq:genus-cycle-formula}
\end{equation}
The number of punctures of the uncompactified cover is
\begin{equation}
b=\sum_s c(\sigma_s),
\end{equation}
which is consistent with
\begin{equation}
2g+b-1=2d+1.
\end{equation}

\begin{theorem}[Higher homotopy after compactification]
Let $\overline X$ be the compact Riemann surface determined by a transitive finite permutation monodromy datum. Then:
\begin{enumerate}
\item if $g\ge1$, $\pi_n(\overline X)=0$ for every $n\ge2$;
\item if $g=0$, then $\overline X\cong S^2$, so $\pi_2(\overline X)\cong\mathbb Z$ and $\pi_n(\overline X)\cong\pi_n(S^2)$ for $n\ge3$.
\end{enumerate}
\end{theorem}

The genus-one Lam\'e elliptic curve \eqref{eq:elliptic-double-cover} belongs to case (i), so its higher homotopy groups vanish despite its nontrivial fundamental lattice $\mathbb Z^2$.

\subsubsection{Chromatic Schreier designs}

The permutation monodromy can be visualized by a colored Schreier graph. Its vertices are the sheets $1,\ldots,d$. For each generator $\gamma_s$, draw an edge of the color assigned to the puncture $s$ from sheet $j$ to $\sigma_s(j)$. Connectedness of the graph is equivalent to transitivity of the covering action, while monochromatic cycles reproduce the cycle decomposition entering \eqref{eq:genus-cycle-formula}.

As an illustrative four-sheet quotient, take
\begin{equation}
\sigma_0=(12)(34),\qquad
\sigma_1=(13)(24),\qquad
\sigma_a=(14)(23),\qquad
\sigma_\infty=1.
\label{eq:klein-permutation-example}
\end{equation}
The first three permutations generate the Klein four-group and obey the required product relation. Their chromatic Schreier design is shown in Fig.~\ref{fig:schreier}. Since
\begin{equation}
c(\sigma_0)=c(\sigma_1)=c(\sigma_a)=2,
\qquad
c(\sigma_\infty)=4,
\end{equation}
formula \eqref{eq:genus-cycle-formula} gives $g=0$ for the compactification, while the punctured four-sheet cover has free fundamental group of rank $9$.

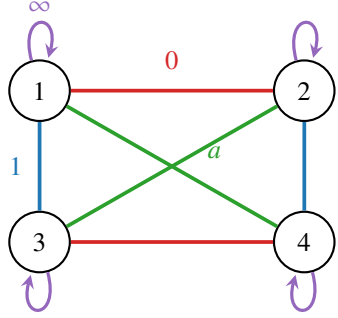
\begin{figure}[t]
\centering
\begin{tikzpicture}[scale=1.25,>=stealth]
  \node[circle,draw,thick,minimum size=8mm] (v1) at (0,1.6) {$1$};
  \node[circle,draw,thick,minimum size=8mm] (v2) at (2.8,1.6) {$2$};
  \node[circle,draw,thick,minimum size=8mm] (v3) at (0,0) {$3$};
  \node[circle,draw,thick,minimum size=8mm] (v4) at (2.8,0) {$4$};
  \draw[loopzero,line width=1.4pt] (v1)--(v2);
  \draw[loopzero,line width=1.4pt] (v3)--(v4);
  \draw[loopone,line width=1.4pt] (v1)--(v3);
  \draw[loopone,line width=1.4pt] (v2)--(v4);
  \draw[loopa,line width=1.4pt] (v1)--(v4);
  \draw[loopa,line width=1.4pt] (v2)--(v3);
  \draw[loopinf,line width=1.2pt,->] (v1) edge[loop above] node[above] {$\infty$} (v1);
  \draw[loopinf,line width=1.2pt,->] (v2) edge[loop above] (v2);
  \draw[loopinf,line width=1.2pt,->] (v3) edge[loop below] (v3);
  \draw[loopinf,line width=1.2pt,->] (v4) edge[loop below] (v4);
  \node[loopzero] at (1.4,1.92) {$0$};
  \node[loopone] at (-0.25,0.8) {$1$};
  \node[loopa] at (1.85,0.95) {$a$};
\end{tikzpicture}
\caption{Chromatic Schreier graph for the four-sheet permutation datum \eqref{eq:klein-permutation-example}. Red, blue, green, and purple encode continuation around $0,1,a,\infty$, respectively. The graph simultaneously records connectivity, cycle data, and the generators used in the van Kampen presentation.}
\label{fig:schreier}
\end{figure}

The construction generalizes immediately to any finite quotient of the linear or projective Heun monodromy group. In computational work it supplies a compact visual invariant of the chosen finite continuation quotient and gives the cycle counts needed for genus and puncture calculations.

\subsubsection{Geometric interpretation of the orbifold and higher homotopy}
\label{subsec:orbifold-homotopy-geometry}

The algebraic statements above admit a useful geometric interpretation.  In particular, the
Lam\'e specialization does not replace the punctured sphere by an arbitrary new surface; it
endows the four-punctured sphere with order-two local projective isotropy.  The resulting
orbifold has signature $(0;2,2,2,2)$ and is commonly visualized as a pillowcase.  A small
neighborhood of each marked point is modeled on a disk modulo the involution
$z\mapsto-z$.  Figure~\ref{fig:pillowcase-orbifold} displays this local-to-global picture.

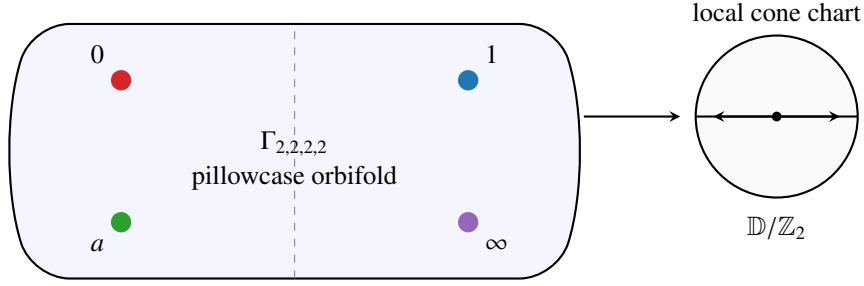
\begin{figure}[t]
\centering
\begin{tikzpicture}[scale=1.02,>=stealth]
  \path[draw=black,thick,fill=blue!4,rounded corners=14pt]
    (-3.4,-1.65) .. controls (-3.75,-0.55) and (-3.75,0.55) .. (-3.4,1.65)
    -- (3.4,1.65)
    .. controls (3.75,0.55) and (3.75,-0.55) .. (3.4,-1.65)
    -- cycle;
  \draw[dashed,gray] (0,-1.65)--(0,1.65);

  \node[fill=loopzero,circle,inner sep=2.7pt,label=above left:$0$]  at (-2.25,0.92) {};
  \node[fill=loopone,circle,inner sep=2.7pt,label=above right:$1$] at (2.25,0.92) {};
  \node[fill=loopa,circle,inner sep=2.7pt,label=below left:$a$]    at (-2.25,-0.92) {};
  \node[fill=loopinf,circle,inner sep=2.7pt,label=below right:$\infty$] at (2.25,-0.92) {};
  \node at (0,0.1) {$\Gamma_{2,2,2,2}$};
  \node at (0,-0.35) {pillowcase orbifold};

  \draw[->,thick] (3.75,0.45)--(5.0,0.45);
  \begin{scope}[shift={(6.25,0.45)}]
    \draw[thick,fill=gray!4] (0,0) circle (1.05);
    \draw[thick] (-1.05,0)--(1.05,0);
    \draw[->,thick] (0,0)--(0.82,0);
    \draw[->,thick] (0,0)--(-0.82,0);
    \fill[black] (0,0) circle (1.8pt);
    \node at (0,-1.45) {$\mathbb D/\mathbb Z_2$};
    \node at (0,1.35) {local cone chart};
  \end{scope}
\end{tikzpicture}
\caption{The Lam\'e projective monodromy orbifold of signature $(0;2,2,2,2)$.  The four colored points are order-two orbifold points.  Locally, each is modeled by the quotient of a disk under the involution $z\mapsto-z$.  The corresponding elliptic curve is an unbranched orbifold double cover of this pillowcase.}
\label{fig:pillowcase-orbifold}
\end{figure}

The orbifold fundamental group is therefore
\begin{equation}
\Gamma_{2,2,2,2}
=
\left\langle
r_0,r_1,r_a,r_\infty
\;\middle|\;
r_0^2=r_1^2=r_a^2=r_\infty^2=1,
\quad r_0r_1r_ar_\infty=1
\right\rangle.
\label{eq:pillowcase-presentation-visual}
\end{equation}
Its index-two translation subgroup is isomorphic to $\mathbb Z^2$ and is the fundamental
group of the elliptic double cover.  Thus the passage from the four-punctured sphere to the
Lam\'e elliptic curve is simultaneously analytic, algebraic, and orbifold-theoretic.

\paragraph{Fundamental group versus higher homotopy groups.}

The phrase ``higher-order fundamental group'' is best replaced by \emph{higher homotopy
group}.  The fundamental group $\pi_1$ classifies based loops, whereas $\pi_n$, $n\ge2$,
classifies based maps of the $n$-sphere into the space.  Figure~\ref{fig:pi1-pi2-comparison}
shows the distinction for $\pi_1$ and $\pi_2$.

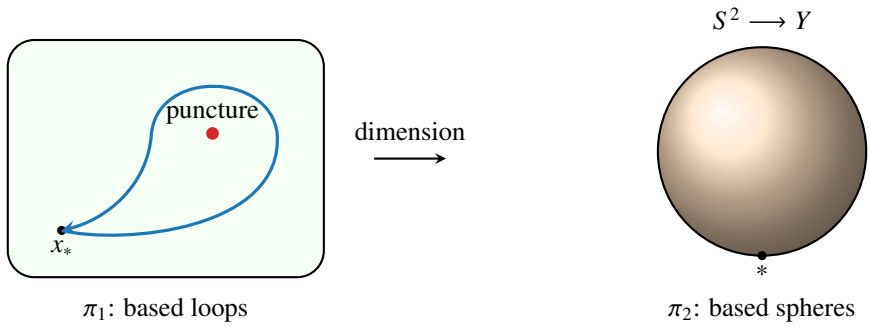
\begin{figure}[t]
\centering
\begin{tikzpicture}[scale=0.95,>=stealth]
  \begin{scope}[shift={(-4.7,0)}]
    \draw[thick,fill=green!4,rounded corners=10pt] (-2.2,-1.65) rectangle (2.2,1.65);
    \fill[black] (-1.45,-1.0) circle (1.8pt);
    \node[below] at (-1.45,-1.0) {$x_*$};
    \fill[loopzero] (0.65,0.35) circle (2.5pt);
    \node[above] at (0.65,0.35) {puncture};
    \draw[loopone,very thick,->]
      (-1.45,-1.0) .. controls (-0.6,-1.2) and (1.65,-1.0) .. (1.55,0.35)
      .. controls (1.45,1.2) and (-0.1,1.25) .. (-0.2,0.35)
      .. controls (-0.25,-0.4) and (-0.85,-0.85) .. (-1.45,-1.0);
    \node at (0,-2.1) {$\pi_1$: based loops};
  \end{scope}

  \draw[->,thick] (-1.8,0)--(-0.8,0);
  \node at (-1.3,0.38) {dimension};

  \begin{scope}[shift={(3.6,0)}]
    \shade[ball color=orange!25] (0,0.1) circle (1.45);
    \draw[thick] (0,0.1) circle (1.45);
    \fill[black] (0,-1.35) circle (1.8pt);
    \node[below] at (0,-1.35) {$*$};
    \node at (0,-2.1) {$\pi_2$: based spheres};
    \node at (0,1.95) {$S^2\longrightarrow Y$};
  \end{scope}
\end{tikzpicture}
\caption{The fundamental group $\pi_1$ records homotopy classes of loops and is generally non-abelian.  The first higher homotopy group $\pi_2$ records homotopy classes of based maps from $S^2$ and is abelian.  Higher groups $\pi_n$, $n\ge2$, probe genuinely higher-dimensional spherical topology.}
\label{fig:pi1-pi2-comparison}
\end{figure}

For the Heun domain
\begin{equation}
X_a=\mathbb{CP}^1\setminus\{0,1,a,\infty\},
\end{equation}
Seifert--van Kampen gives $\pi_1(X_a)\cong F_3$.  On the other hand, $X_a$ is homotopy
equivalent to a graph, hence is aspherical.  Therefore
\begin{equation}
\pi_n(X_a)=0,
\qquad n\ge2.
\label{eq:punctured-higher-vanish}
\end{equation}
The same conclusion holds for every connected unbranched monodromy cover
$\widetilde X_q\to X_a$, as already proved in Theorem~\ref{thm:monodromy-cover}.

\paragraph{Criterion for nonvanishing higher homotopy.}

The relevant structural condition is failure of asphericity.  A connected space $Y$ is
aspherical when its universal cover is contractible, equivalently
\begin{equation}
\pi_n(Y)=0,
\qquad n\ge2.
\end{equation}
Thus a higher homotopy group can be nonzero only when the universal cover carries
nontrivial higher-dimensional homotopy.  In the two-dimensional compactified covering
problems considered here, this criterion becomes especially sharp.

\begin{proposition}[Nonvanishing criterion for compactified Heun covers]
\label{prop:nonvanishing-higher}
Let $\overline X$ be a connected compact Riemann surface obtained by compactifying a finite
permutation-monodromy cover of $X_a$, and let $g$ be its genus as determined by
\eqref{eq:genus-cycle-formula}.  Then
\begin{equation}
\pi_n(\overline X)=0,
\qquad n\ge2,
\end{equation}
if $g\ge1$.  If $g=0$, then $\overline X\cong S^2$ and
\begin{equation}
\pi_2(\overline X)\cong\mathbb Z,
\end{equation}
while for $n\ge3$ one has
\begin{equation}
\pi_n(\overline X)\cong\pi_n(S^2),
\end{equation}
which is nonzero for infinitely many $n$.
\end{proposition}

\begin{proof}
A closed orientable surface of genus $g\ge1$ has contractible universal cover: the Euclidean
plane for $g=1$ and the hyperbolic disk for $g\ge2$.  Hence it is aspherical.  For $g=0$ the
surface is $S^2$, whose second homotopy group is $\mathbb Z$ and whose higher homotopy
groups are the classical sphere homotopy groups.
\end{proof}

Consequently, within the present Heun covering theory the first obstruction to vanishing is
not merely the presence of branching; it is whether compactification and branching change
the covering surface into a non-aspherical topology.  For compact Riemann surfaces, this
happens precisely at genus zero.  The genus-one Lam\'e elliptic cover remains aspherical even
though its fundamental group is nontrivial:
\begin{equation}
\pi_1(E_a)\cong\mathbb Z^2,
\qquad
\pi_n(E_a)=0,
\quad n\ge2.
\end{equation}
By contrast, the four-sheet example in \eqref{eq:klein-permutation-example} compactifies to
genus zero; hence its compactification has $\pi_2\cong\mathbb Z$.

Figure~\ref{fig:aspherical-versus-sphere} summarizes the distinction.

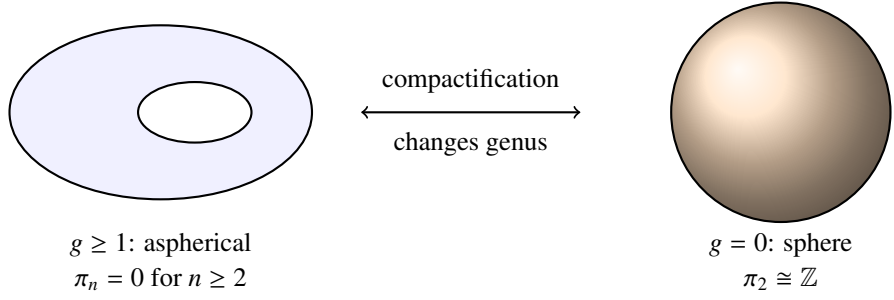
\begin{figure}[t]
\centering
\begin{tikzpicture}[scale=1.0]
  \begin{scope}[shift={(-4.1,0)}]
    \draw[thick,fill=blue!6] (0,0) ellipse (2.0 and 1.15);
    \draw[thick,fill=white] (0.45,0) ellipse (0.75 and 0.40);
    \node at (0,-1.75) {$g\ge1$: aspherical};
    \node at (0,-2.2) {$\pi_n=0$ for $n\ge2$};
  \end{scope}

  \draw[<->,thick] (-1.45,0)--(1.45,0);
  \node at (0,0.42) {compactification};
  \node at (0,-0.42) {changes genus};

  \begin{scope}[shift={(4.1,0)}]
    \shade[ball color=orange!25] (0,0) circle (1.45);
    \draw[thick] (0,0) circle (1.45);
    \node at (0,-1.75) {$g=0$: sphere};
    \node at (0,-2.2) {$\pi_2\cong\mathbb Z$};
  \end{scope}
\end{tikzpicture}
\caption{Higher homotopy after compactification.  Compact Riemann surfaces of genus $g\ge1$ are aspherical, whereas genus zero gives the sphere and a nonvanishing second homotopy group.}
\label{fig:aspherical-versus-sphere}
\end{figure}

This distinction is useful for the spectral problem.  The monodromy representation is governed
primarily by $\pi_1$, because analytic continuation follows loops.  The higher homotopy groups
do not supply additional ordinary monodromy matrices for a second-order Fuchsian equation;
rather, they describe the global topology of compactified continuation surfaces.  They become
relevant when one studies families of covers, compactifications, topological transitions, or
higher-geometric structures attached to the Heun local system.

\section{Main Results}
\label{sec:main-results}

The results in this section are the contributions developed from the Lam\'e--Heun framework
of Section~\ref{sec:preliminaries}. Established theorems are cited where they are invoked; the
statements below concern the specific parameter, spectral, asymptotic, and monodromy
consequences derived for the Lam\'e locus.

\subsection{Lam\'e-preserving Heun parameter symmetries}
\label{sec:symmetry}
For brevity write $p_0=\alpha\beta$. Two generators are
\begin{align}
s:&\quad (a,q)\mapsto(1-a,p_0-q),\label{eq:sgen}\\
t:&\quad (a,q)\mapsto(a^{-1},q/a).\label{eq:tgen}
\end{align}
They satisfy $s^2=t^2=(st)^3=1$.

\begin{theorem}[Explicit Lam\'e-preserving anharmonic orbit]
For fixed $\alpha,\beta$ and $\gamma=\delta=\epsilon=1/2$, the six $S_3$ parameter transforms are
\begin{center}
\begin{tabular}{ccc}
\toprule
No. & $a'$ & $q'$\\
\midrule
1 & $a$ & $q$\\
2 & $1-a$ & $p_0-q$\\
3 & $1/a$ & $q/a$\\
4 & $1/(1-a)$ & $(p_0-q)/(1-a)$\\
5 & $(a-1)/a$ & $(ap_0-q)/a$\\
6 & $a/(a-1)$ & $(ap_0-q)/(a-1)$\\
\bottomrule
\end{tabular}
\end{center}
Each transformed equation is again a Lam\'e-type Heun equation with finite exponent parameters $1/2,1/2,1/2$.
\end{theorem}

\begin{proof}
The transformations \eqref{eq:sgen} and \eqref{eq:tgen} follow by direct substitution of $z\mapsto1-z$ and $z\mapsto z/a$ into \eqref{eq:heun}. Their compositions generate the six elements of $S_3$, giving the table.
\end{proof}

Substituting \eqref{eq:lame-heun-parameters}, so that $p_0=-\nu(\nu+1)/4$, produces explicit transformations of the modulus and Lam\'e spectral parameter. Write $m=k^2$. Since $a'=1/m'$ and $q'=-h'/(4m')$, one obtains
\begin{center}
\begin{tabular}{ccc}
\toprule
No. & $m'$ & $h'$\\
\midrule
1 & $m$ & $h$\\
2 & $\displaystyle\frac{m}{m-1}$ & $\displaystyle\frac{-h+m\nu(\nu+1)}{m-1}$\\[2mm]
3 & $1/m$ & $h/m$\\[1mm]
4 & $1-1/m$ & $\nu(\nu+1)-h/m$\\[1mm]
5 & $1/(1-m)$ & $\displaystyle\frac{h-\nu(\nu+1)}{m-1}$\\[2mm]
6 & $1-m$ & $-h+\nu(\nu+1)$\\
\bottomrule
\end{tabular}
\end{center}
These maps should be understood as analytic parameter equivalences; they need not preserve the real physical interval $0<m<1$.

\begin{theorem}[Orbit stratification]
The anharmonic orbit of a generic singularity parameter $a\in\C\setminus\{0,1\}$ contains six distinct values
\begin{equation}
a,\quad 1-a,\quad a^{-1},\quad (1-a)^{-1},\quad (a-1)/a,\quad a/(a-1).
\end{equation}
The orbit has three elements at the harmonic values
\begin{equation}
a\in\{-1,\tfrac12,2\},
\end{equation}
and two elements at the equianharmonic values
\begin{equation}
a=\frac{1\pm i\sqrt3}{2}.
\end{equation}
Hence the corresponding stabilizers in $S_3$ have orders $1$, $2$, and $3$, respectively.
\end{theorem}

\begin{proof}
The statement follows from the orbit--stabilizer theorem applied to the six cross-ratio transforms. Coincidences occur precisely when $a$ is fixed by a nontrivial transposition or a $3$-cycle. Solving the resulting fixed-point equations gives the harmonic and equianharmonic sets above.
\end{proof}

For a real Jacobi modulus $0<m<1$, the harmonic point inside the physical interval is $m=1/2$, corresponding to $a=2$. The equianharmonic values are necessarily complex.

\begin{corollary}[Spectral covariance]
Suppose $h$ is an admissible Lam\'e spectral value at modulus $m$ for a boundary or monodromy problem transported by one of the six transformations above. Then the corresponding transformed value $h'$ in the table is admissible for the transformed modulus $m'$, provided the boundary or monodromy data are transported by the same change of variable.
\end{corollary}

\begin{remark}
The corollary is a covariance statement, not an assertion that a fixed real boundary-value problem has identical spectra at all six parameter values. The domain and boundary data must be transformed together with the differential equation.
\end{remark}

\subsection{Finite polynomial sectors and accessory spectral polynomials}
\label{sec:qes}

The Heun embedding also yields a finite-dimensional spectral problem. Multiply \eqref{eq:heun} by $z(z-1)(z-a)$ and write
\begin{align}
\mathcal H={}&[z^3-(1+a)z^2+az]\partial_z^2
\nonumber\\
&+[Sz^2-Tz+a\gamma]\partial_z+\alpha\beta z-q,
\label{eq:heun-poly-operator}
\end{align}
where
\begin{equation}
S=\gamma+\delta+\epsilon,
\qquad
T=\gamma(1+a)+a\delta+\epsilon.
\end{equation}
On $z^r$,
\begin{align}
\mathcal H z^r={}&(r+\alpha)(r+\beta)z^{r+1}
\nonumber\\
&-[(1+a)r(r-1)+Tr+q]z^r
+a r(r+\gamma-1)z^{r-1}.
\label{eq:monomial-action}
\end{align}

For the Lam\'e parameters, $\gamma=\delta=\epsilon=1/2$, hence $S=3/2$ and $T=1+a$.

\begin{theorem}[Even Lam\'e polynomial sector]
Let $\nu=2N$, $N\in\Z_{\ge0}$. Then
\begin{equation}
\alpha=-N,
\qquad
\beta=N+\frac12,
\end{equation}
and the polynomial space
\begin{equation}
\Pn=\operatorname{span}\{1,z,\ldots,z^N\}
\end{equation}
is invariant under the polynomial Heun operator \eqref{eq:heun-poly-operator}. In the monomial basis, the accessory matrix $L_N$ has entries
\begin{align}
A_r&=(r-N)\left(r+N+\frac12\right),\label{eq:Ar}\\
B_r&=-(1+a)r^2,\label{eq:Br}\\
C_r&=a r\left(r-\frac12\right),\label{eq:Cr}
\end{align}
with
\begin{equation}
L_N=
\begin{pmatrix}
B_0&C_1&0&\cdots&0\\
A_0&B_1&C_2&\ddots&\vdots\\
0&A_1&B_2&\ddots&0\\
\vdots&\ddots&\ddots&\ddots&C_N\\
0&\cdots&0&A_{N-1}&B_N
\end{pmatrix}.
\end{equation}
Polynomial Lam\'e--Heun solutions occur exactly when
\begin{equation}
\Delta_N(q;a)=\det(L_N-qI)=0.
\label{eq:spectral-poly}
\end{equation}
\end{theorem}

\begin{proof}
Equation \eqref{eq:monomial-action} raises degree by one with coefficient $(r+\alpha)(r+\beta)$. At $r=N$, this coefficient vanishes because $N+\alpha=0$, so $\Pn$ is invariant. Substituting the Lam\'e parameters into \eqref{eq:monomial-action} gives \eqref{eq:Ar}--\eqref{eq:Cr}. The accessory parameter is therefore an eigenvalue of $L_N$.
\end{proof}

The determinant obeys the three-term recurrence
\begin{equation}
\Delta_{-1}=1,
\qquad
\Delta_0=-q,
\end{equation}
\begin{equation}
\Delta_r=(B_r-q)\Delta_{r-1}-A_{r-1}C_r\Delta_{r-2},
\qquad r=1,\ldots,N.
\label{eq:det-rec}
\end{equation}
Since $q=-h/(4m)$ and $a=1/m$, the roots of \eqref{eq:spectral-poly} give the corresponding finite Lam\'e spectral values.

\begin{remark}
For odd integral $\nu$, polynomial sectors occur after multiplication by one or more local exponent factors $z^{1/2}$, $(z-1)^{1/2}$, and $(z-a)^{1/2}$. These gauge sectors are the natural continuation of the same Heun truncation mechanism but are not required for the even-sector result above.
\end{remark}

\subsubsection{Example: $\nu=4$}

For $\nu=4$ one has $N=2$ and a $3\times3$ accessory matrix. Figure~\ref{fig:bands} plots the three values of $h=-4mq$ obtained from the exact eigenvalues of $L_2$ as $m$ varies in the physical interval. The figure is illustrative: it is generated directly from the finite matrix, not from fitted data.

\begin{figure}[htbp]
\centering
\begin{tikzpicture}
\begin{axis}[
width=0.82\textwidth,
height=0.48\textwidth,
xlabel={$m=k^2$},
ylabel={$h$},
xmin=0.05,xmax=0.95,
ymin=0,ymax=21,
grid=major,
legend pos=north west]
\addplot[thick] table[row sep=\\] {
m h\\
0.050000 0.471356\\0.100000 0.884925\\0.150000 1.243419\\0.200000 1.553129\\0.250000 1.821995\\0.300000 2.057877\\0.350000 2.267601\\0.400000 2.456672\\0.450000 2.629355\\0.500000 2.788897\\0.550000 2.937756\\0.600000 3.077796\\0.650000 3.210442\\0.700000 3.336790\\0.750000 3.457697\\0.800000 3.573835\\0.850000 3.685741\\0.900000 3.793847\\0.950000 3.898504\\};
\addlegendentry{$h_1$}
\addplot[thick,dashed] table[row sep=\\] {
m h\\
0.050000 4.427148\\0.100000 4.908922\\0.150000 5.442322\\0.200000 6.020706\\0.250000 6.635702\\0.300000 7.278913\\0.350000 7.942840\\0.400000 8.621124\\0.450000 9.308401\\0.500000 10.000000\\0.550000 10.691599\\0.600000 11.378876\\0.650000 12.057160\\0.700000 12.721087\\0.750000 13.364298\\0.800000 13.979294\\0.850000 14.557678\\0.900000 15.091078\\0.950000 15.572852\\};
\addlegendentry{$h_2$}
\addplot[thick,dotted] table[row sep=\\] {
m h\\
0.050000 16.101496\\0.100000 16.206153\\0.150000 16.314259\\0.200000 16.426165\\0.250000 16.542303\\0.300000 16.663210\\0.350000 16.789558\\0.400000 16.922204\\0.450000 17.062244\\0.500000 17.211103\\0.550000 17.370645\\0.600000 17.543328\\0.650000 17.732399\\0.700000 17.942123\\0.750000 18.178005\\0.800000 18.446871\\0.850000 18.756581\\0.900000 19.115075\\0.950000 19.528644\\};
\addlegendentry{$h_3$}
\end{axis}
\end{tikzpicture}
\caption{Three finite polynomial spectral branches for $\nu=4$ ($N=2$), computed from the exact accessory matrix $L_2$.}
\label{fig:bands}
\end{figure}
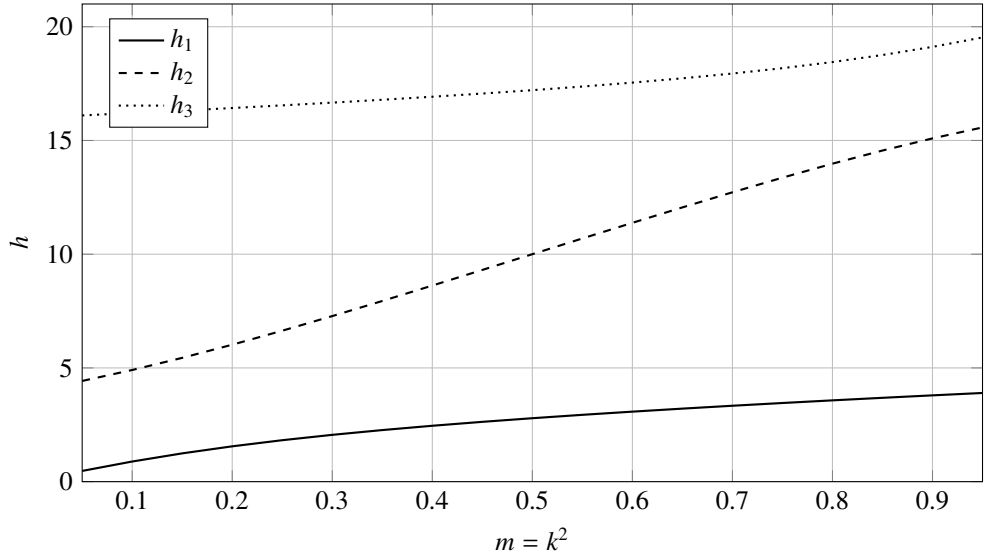

\subsection{Coefficient asymptotics and accessory spectral selection}
\label{sec:complete-asymptotic}
\subsubsection{Poincar\'e roots and exponential coefficient growth}

Divide \eqref{eq:heun-recurrence-general} by $n^2$ and let $n\to\infty$. The limiting recurrence is
\begin{equation}
a c_{n+1}-(1+a)c_n+c_{n-1}=0.
\label{eq:limiting-recurrence}
\end{equation}
Its characteristic polynomial is
\begin{equation}
a r^2-(1+a)r+1=(r-1)(ar-1),
\label{eq:char-poly}
\end{equation}
so the two Poincar\'e roots are
\begin{equation}
r_1=1,
\qquad
r_a=\frac1a.
\label{eq:poincare-roots}
\end{equation}
For the Lam\'e specialization they are $1$ and $m$. Consequently, the recurrence itself already encodes the two finite continuation barriers $z=1$ and $z=a$: the reciprocal moduli of the characteristic roots are their distances from the expansion center.

\begin{proposition}[Generic root growth]
\label{prop:generic-root-growth}
Assume $|a|>1$ and that the exponent-zero solution at $z=0$ has a nonzero singular connection coefficient at $z=1$. Then
\begin{equation}
\limsup_{n\to\infty}|c_n(q)|^{1/n}=1.
\label{eq:generic-limsup}
\end{equation}
If the $z=1$ singular contribution vanishes and the next obstruction is the singular point $z=a$, then
\begin{equation}
\limsup_{n\to\infty}|c_n(q)|^{1/n}=\frac1{|a|}.
\label{eq:spectral-limsup}
\end{equation}
For $a=1/m$, $0<m<1$, the second rate is $m$.
\end{proposition}

\begin{proof}
The Cauchy--Hadamard formula gives
\begin{equation}
R^{-1}=\limsup_{n\to\infty}|c_n|^{1/n}.
\end{equation}
Generically the nearest nonremovable singularity is at $z=1$, so $R=1$. If the corresponding local singular component is absent, the solution extends holomorphically through $z=1$ and the next finite obstruction lies at $z=a$, giving $R=|a|$. The two values coincide with the characteristic roots of \eqref{eq:limiting-recurrence}.
\end{proof}

\subsubsection{Darboux asymptotics and the accessory spectrum}

The preceding root test determines only the exponential scale. The algebraic prefactor is obtained from the local Frobenius exponents. Near $z=1$, write the continuation of the normalized solution \eqref{eq:heun-series} as
\begin{equation}
y(z;q)=A_0(q)\,f_0(z;q)
+A_1(q)(1-z)^{1-\delta}f_1(z;q),
\label{eq:connection-at-one}
\end{equation}
where $f_0$ and $f_1$ are analytic and nonzero at $z=1$. The coefficient $A_1(q)$ is the connection coefficient of the branch that is nonanalytic in the algebraic variable whenever $1-\delta\notin\Z_{\ge0}$.

By Darboux's principle \cite{KnuthWilf}, if $A_1(q)\neq0$ and $z=1$ is the unique nearest singularity, then
\begin{equation}
c_n(q)=
\frac{A_1(q)f_1(1;q)}{\Gamma(\delta-1)}
\,n^{\delta-2}
\left(1+O\left(n^{-1}\right)\right).
\label{eq:darboux-general-one}
\end{equation}
The corresponding contribution from $z=a$ has the form
\begin{equation}
c_n^{(a)}(q)=
\frac{A_a(q)g_a(a;q)}{\Gamma(\epsilon-1)}
\,a^{-n}n^{\epsilon-2}
\left(1+O\left(n^{-1}\right)\right),
\label{eq:darboux-general-a}
\end{equation}
where $A_a(q)$ is the singular connection coefficient at $z=a$.

For the Lam\'e parameters $\delta=\epsilon=1/2$, these simplify to
\begin{equation}
c_n(q)=
-\frac{A_1(q)f_1(1;q)}{2\sqrt\pi}
\,n^{-3/2}
+O(n^{-5/2})
\label{eq:lame-darboux-one}
\end{equation}
when the singularity at $z=1$ is present, whereas after its removal the next contribution is
\begin{equation}
c_n(q)=
-\frac{A_a(q)g_a(a;q)}{2\sqrt\pi}
\,a^{-n}n^{-3/2}
\left(1+O\left(n^{-1}\right)\right).
\label{eq:lame-darboux-a}
\end{equation}
Thus, in the physical Lam\'e range,
\begin{equation}
c_n(q)=O\left(m^n n^{-3/2}\right)
\label{eq:lame-spectral-decay}
\end{equation}
whenever the $z=1$ singular branch has been removed and $z=a=1/m$ is the next obstruction.

\begin{theorem}[Accessory spectrum as asymptotic branch suppression]
\label{thm:accessory-asymptotic}
Fix an endpoint parity sector of the Lam\'e problem and factor the corresponding local powers
\begin{equation}
Y(z)=z^{\sigma_0}(1-z)^{\sigma_1}F(z),
\qquad
\sigma_0,\sigma_1\in\left\{0,\frac12\right\},
\label{eq:parity-gauge}
\end{equation}
so that the physical eigenfunction is represented by a function $F$ analytic at both $z=0$ and $z=1$. Let $F(z;q)=\sum_{n\ge0}d_n(q)z^n$ denote the corresponding normalized Frobenius series at $z=0$. Then the admissible accessory values $q_j$ in that parity sector are precisely the zeros of the unwanted connection coefficient at $z=1$:
\begin{equation}
\mathcal C_{\sigma_0,\sigma_1}(q_j)=0.
\label{eq:connection-spectrum}
\end{equation}
At a generic value of $q$,
\begin{equation}
\limsup_{n\to\infty}|d_n(q)|^{1/n}=1,
\end{equation}
whereas at an eigenvalue $q_j$, provided $z=a$ is the next nonremovable singularity,
\begin{equation}
\limsup_{n\to\infty}|d_n(q_j)|^{1/n}=\frac1{|a|}=m.
\label{eq:qj-root-drop}
\end{equation}
Moreover, for the Lam\'e exponent $1/2$ at the next singularity,
\begin{equation}
d_n(q_j)=D_j\,m^n n^{-3/2}
\left(1+O\left(n^{-1}\right)\right)
\label{eq:qj-asymptotic}
\end{equation}
with a nonzero constant $D_j$ unless the singularity at $z=a$ is also removable.
\end{theorem}

\begin{proof}
For a fixed parity sector, the factorization \eqref{eq:parity-gauge} converts the desired endpoint exponent at $z=1$ into the analytic exponent of the reduced equation. The boundary-value eigencondition is therefore equivalent to the absence of the complementary local Frobenius branch, which is the scalar connection condition \eqref{eq:connection-spectrum}. For nonzero connection coefficient, the nearest singular contribution at $z=1$ gives unit root growth by Darboux's theorem and Cauchy--Hadamard. At $q=q_j$ that contribution vanishes, so analytic continuation crosses $z=1$ and the next obstruction is $z=a$. The Cauchy--Hadamard radius becomes $|a|$, while the local exponent $1/2$ at $z=a$ gives the $n^{-3/2}$ Darboux prefactor. Since $a=1/m$, the exponential factor is $m^n$.
\end{proof}

This theorem gives a direct spectral interpretation of large-order coefficients: the accessory eigenvalues are exactly those values for which the dominant asymptotic branch is suppressed. Numerically, one may therefore detect $q_j$ either by a connection determinant, by a continued fraction generated from \eqref{eq:heun-recurrence-general}, or by observing the root-growth transition of $d_n(q)$ from $1$ to $m$.

\subsubsection{Polynomial eigenvalues as exact termination}

The finite spectral values of Section~\ref{sec:qes} are stronger exceptional points. When $\alpha=-N$ (or, after an appropriate exponent gauge, $\beta=-N$) and
\begin{equation}
\Delta_N(q_j;a)=0,
\end{equation}
the recurrence terminates:
\begin{equation}
c_{N+1}(q_j)=c_{N+2}(q_j)=\cdots=0.
\label{eq:termination}
\end{equation}
Hence the corresponding algebraic Heun eigenfunction is a polynomial and has infinite radius of convergence as a function of $z$. In coefficient language,
\begin{equation}
\limsup_{n\to\infty}|c_n(q_j)|^{1/n}=0.
\label{eq:poly-root-growth}
\end{equation}
Thus the accessory spectrum displays a three-level analytic hierarchy:
\begin{equation}
\text{generic }q:\rho=1,
\qquad
\text{regular spectral }q_j:\rho=m,
\qquad
\text{polynomial }q_j:\rho=0,
\label{eq:three-level-growth}
\end{equation}
where $\rho=\limsup |c_n|^{1/n}$ after the parity-adapted gauge. This gives a coefficient-asymptotic signature that distinguishes generic local Heun solutions, globally admissible Lam\'e eigenfunctions, and exact finite polynomial states.

\subsubsection{Complex-analytic interpretation}

The numerical radius of one Taylor series is not invariant under a M\"obius transformation of the Heun singular set. The invariant object is the four-point configuration
\begin{equation}
\{0,1,a,\infty\}\subset\CP^1
\end{equation}
together with its cross-ratio. For the Lam\'e family $a=1/m$, the elliptic modulus therefore controls both the conformal singularity geometry and the exponential scale of the parity-adapted Taylor coefficients. The spectral condition removes one local branch without changing the differential equation's singular set; in this sense spectral quantization appears as a selection of a more analytically continuable solution inside a fixed Fuchsian geometry.


\subsection{Lam\'e monodromy, character geometry, and spectral alignment}
\label{sec:monodromy}
\subsubsection{Lam\'e specialization and the pillowcase orbifold}

For the Lam\'e--Heun parameters
\begin{equation}
\gamma=\delta=\epsilon=\frac12,
\qquad
\alpha=-\frac{\nu}{2},
\qquad
\beta=\frac{\nu+1}{2},
\end{equation}
the three finite local monodromies have eigenvalues $\{1,-1\}$. If $\nu\in\mathbb Z$, the monodromy at infinity has the same unordered eigenvalue set. Hence every local projective monodromy is an involution.

\begin{theorem}[Lam\'e projective monodromy factorization]
\label{thm:pillow}
For integer Lam\'e coupling $\nu\in\mathbb Z$, the projective monodromy representation factors through the orbifold group of signature $(0;2,2,2,2)$,
\begin{equation}
\Gamma_{2,2,2,2}
=
\left\langle r_0,r_1,r_a,r_\infty\;\middle|\;
 r_0^2=r_1^2=r_a^2=r_\infty^2=1,
\quad r_0r_1r_ar_\infty=1
\right\rangle.
\label{eq:pillow-group}
\end{equation}
Moreover,
\begin{equation}
\Gamma_{2,2,2,2}\cong \mathbb Z^2\rtimes_{-1}\mathbb Z_2,
\end{equation}
and its index-two translation subgroup is isomorphic to $\mathbb Z^2$.
\end{theorem}

\begin{proof}
The exponent differences at the four singular points are $1/2,1/2,1/2,\nu+1/2$. For integer $\nu$, each local projective monodromy therefore has order two. Relation \eqref{eq:loop-relation} supplies the product relation, so the representation factors through \eqref{eq:pillow-group}. The group in \eqref{eq:pillow-group} is the Euclidean pillowcase orbifold group. It is generated by half-turns and contains the translation lattice of the universal Euclidean covering as an index-two normal subgroup, yielding the stated semidirect product. This is precisely the orbifold structure underlying the twofold elliptic covering of the four branch points.
\end{proof}

The geometric content is immediate. The elliptic curve
\begin{equation}
E_a:\quad y^2=z(z-1)(z-a)
\label{eq:elliptic-double-cover}
\end{equation}
is a twofold branched cover of $\mathbb{CP}^1$ at $0,1,a,\infty$. Interpreted orbifold-theoretically, it is an unbranched cover of the pillowcase orbifold, and its fundamental group $\pi_1(E_a)\cong\mathbb Z^2$ is the translation subgroup in Theorem~\ref{thm:pillow}. Thus the classical elliptic uniformization of Lam\'e's equation is also the natural topological cover on which the order-two local projective branching is resolved.

\subsubsection{$SL_2(\mathbb C)$ character surface}

Rescale each local monodromy by a scalar so that
\begin{equation}
\widehat M_s\in SL_2(\mathbb C).
\end{equation}
Let
\begin{equation}
p_s=\operatorname{tr}(\widehat M_s),
\end{equation}
and introduce the composite trace coordinates
\begin{equation}
x=\operatorname{tr}(\widehat M_0\widehat M_1),
\qquad
y=\operatorname{tr}(\widehat M_1\widehat M_a),
\qquad
z=\operatorname{tr}(\widehat M_0\widehat M_a).
\end{equation}
The four-punctured-sphere character variety is the Fricke cubic
\begin{align}
0={}&x^2+y^2+z^2+xyz
-(p_0p_1+p_ap_\infty)x
-(p_1p_a+p_0p_\infty)y
\nonumber\\
&-(p_0p_a+p_1p_\infty)z
+p_0^2+p_1^2+p_a^2+p_\infty^2
+p_0p_1p_ap_\infty-4.
\label{eq:fricke-cubic}
\end{align}

For integer Lam\'e coupling the normalized local traces vanish, so the character surface reduces to
\begin{equation}
x^2+y^2+z^2+xyz-4=0.
\label{eq:lame-fricke}
\end{equation}
Thus the accessory parameter defines a holomorphic, possibly multivalued Riemann--Hilbert map
\begin{equation}
q\longmapsto \bigl(x(q),y(q),z(q)\bigr)
\end{equation}
into the cubic surface \eqref{eq:lame-fricke}. The dependence of global Lam\'e monodromy on the accessory spectrum may therefore be studied as a curve on a fixed character surface rather than as unrelated matrix data.

\subsubsection{Monodromy alignment and the accessory spectrum}

Let $Y_0(z;q)$ be the exponent-zero solution based at $z=0$, and let
\begin{equation}
\bigl(Y_1^{(0)},Y_1^{(1-\delta)}\bigr)
\end{equation}
be a Frobenius basis at $z=1$. On an overlap domain,
\begin{equation}
Y_0(z;q)
=A_0(q)Y_1^{(0)}(z;q)+A_1(q)Y_1^{(1-\delta)}(z;q).
\label{eq:connection-coefficients}
\end{equation}
The coefficient $A_1(q)$ is the same singular connection amplitude that controls the leading Darboux contribution to the Taylor coefficients developed in Section~6.

\begin{theorem}[Spectral regularity as monodromy alignment]
\label{thm:alignment}
Assume the endpoint condition requires the continuation of $Y_0$ to lie in the regular exponent line at $z=1$. Then an accessory value $q_j$ is admissible precisely when
\begin{equation}
A_1(q_j)=0.
\label{eq:alignment-zero}
\end{equation}
Equivalently, the analytically continued line spanned by $Y_0$ is an eigenline of the local monodromy $M_1$ associated with the regular exponent. At such a value the nearest singular contribution to the Taylor coefficients is suppressed, and the coefficient asymptotics are governed by the next nonremovable singularity.
\end{theorem}

\begin{proof}
Equation \eqref{eq:connection-coefficients} decomposes the transported solution into the two local monodromy eigendirections at $z=1$. The imposed endpoint regularity removes the nonregular Frobenius branch, hence is equivalent to \eqref{eq:alignment-zero}. The same coefficient multiplies the corresponding singular term in the local Darboux expansion; its vanishing removes that contribution and exposes the next singularity in the coefficient growth. This proves the claimed equivalence between boundary admissibility, monodromy-line alignment, and the change in large-order coefficient asymptotics.
\end{proof}

For periodic or antiperiodic Lam\'e problems on a real period cell, the analogous condition is expressed by a composite monodromy or transfer matrix $T(E)$ satisfying
\begin{equation}
\operatorname{tr}T(E)=2
\qquad\text{or}\qquad
\operatorname{tr}T(E)=-2,
\end{equation}
respectively. Thus the band edges are trace-level sets of the same global monodromy representation that controls complex analytic continuation.


\section{Conclusion}
\label{sec:conclusion}

\subsection{Discussion of results}
The analysis separates three levels that are often conflated. First, the transformations from ellipsoidal to algebraic, Weierstrass, and Jacobi forms are coordinate or uniformization transformations. Second, the six maps in Section~\ref{sec:symmetry} are parameter equivalences inherited from the finite-singularity permutation subgroup of the Heun group. Third, spectral equivalence requires the accessory parameter, modulus, and boundary or monodromy data to be transported together.

The extension through the accessory matrix adds a finite-dimensional spectral layer. For $\nu=2N$, the un-gauged Lam\'e--Heun operator has an $(N+1)$-dimensional invariant polynomial module. The roots of $\Delta_N$ are therefore exact algebraic spectral data. The six anharmonic transformations act on the same parameter family, so a natural next step is to derive the induced transformation law of $\Delta_N$ itself and to identify its invariant combinations. For odd $\nu$, the corresponding gauge-polynomial sectors should be treated simultaneously; doing so would recover the full family of classical Lam\'e polynomial sectors within a single Heun-theoretic framework.

The monodromy analysis now makes the elliptic covering precise. For integer coupling, the projective local monodromies are involutions and the representation factors through the pillowcase orbifold group $\Gamma_{2,2,2,2}$. Its index-two translation subgroup is the lattice fundamental group of the elliptic double cover. The $SL_2$ character variety reduces to the cubic $x^2+y^2+z^2+xyz-4=0$, so the accessory parameter traces a Riemann--Hilbert curve on a fixed complex surface. This gives a concrete bridge between local Heun continuation and the global elliptic geometry of the Lam\'e operator.

The covering-space layer is equally rigid. Seifert--van Kampen gives $\pi_1(X_a)\cong F_3$; every connected $d$-sheeted unbranched cover therefore has free fundamental group of rank $2d+1$, and all its higher homotopy groups vanish. For finite permutation quotients, compactification converts cycle data of the colored monodromy generators directly into the genus by Riemann--Hurwitz. The chromatic Schreier graphs in Section~\ref{sec:coverings} package these algebraic and topological data in a form suitable for symbolic or numerical continuation.

The completeness and coefficient analysis adds a fourth level: analytic admissibility. The same accessory parameter that enters the three-term recurrence controls the connection coefficient at the next singular point. Consequently the physical spectrum can be read from the large-order behavior of the Frobenius coefficients. Generic parameters retain the $z=1$ branch and have unit root growth; parity-adapted eigenvalues suppress that branch and expose the $z=a$ scale $m$; polynomial eigenvalues terminate. This coefficient signature is independent of any finite truncation used to compute the spectrum and therefore supplies a useful asymptotic diagnostic for numerical continuation and spectral verification.

We have tightened the classical transformation theory of the Lam\'e equation and extended it in two directions. The geometric derivation from confocal ellipsoidal coordinates leads to the algebraic equation with singularities at the three cubic roots and the single point at infinity. Weierstrass uniformization converts this equation into an elliptic Schr\"odinger problem, while the Jacobi form follows only after the correct scaling and imaginary half-period translation. The substitution $z=\sn^2(u\mid k^2)$ then embeds the Lam\'e problem into the canonical Heun family with
\begin{equation}
a=k^{-2},\qquad
q=-\frac{h}{4k^2},\qquad
\alpha=-\frac{\nu}{2},\qquad
\beta=\frac{\nu+1}{2},\qquad
\gamma=\delta=\epsilon=\frac12.
\end{equation}
Within this Lam\'e locus, the finite-singularity permutation subgroup of the Heun transformation group acts as the six-element anharmonic group. We derived its complete action on $(a,q)$ and on $(m,h)$, and identified the harmonic and equianharmonic orbit degeneracies. For even integral coupling $\nu=2N$, we further obtained the invariant polynomial space, tridiagonal accessory matrix, determinant recurrence, and finite spectral polynomial. These results place classical Lam\'e transformations, Heun parameter symmetry, and exact finite spectral data in one framework. In addition, the self-adjoint Lam\'e realizations on a period cell possess complete eigenfunction systems, while their Heun representatives reveal a sharp complex-analytic spectral signature. After the appropriate endpoint-parity gauge, generic Frobenius coefficients have unit root growth, admissible nonpolynomial eigenvalues satisfy $|d_n|^{1/n}\to m$ in the limsup sense, and exact polynomial states terminate. The accessory spectrum can therefore be interpreted as the zero set of a connection coefficient whose vanishing suppresses the dominant Darboux contribution and enlarges the analytic continuation disk from the nearest singularity $z=1$ to $z=1/m$.

The global continuation theory adds a topological interpretation of the same spectral data. The four local monodromy generators satisfy the van Kampen relation and generate an accessory-dependent representation of the free group $F_3$. For integer Lam\'e coupling their projective classes are involutions, forcing the representation through the $(2,2,2,2)$ pillowcase orbifold group and explaining the elliptic double cover as its translation subgroup. After $SL_2$ normalization, the Lam\'e monodromy characters lie on the cubic surface $x^2+y^2+z^2+xyz-4=0$. Spectral regularity is equivalent to alignment of the transported Frobenius line with a local monodromy eigendirection, which is exactly the condition that removes the leading singular coefficient asymptotic. Finally, monodromy coverings are aspherical: their higher homotopy groups vanish, finite-sheeted covers have fundamental-group rank $2d+1$, and finite permutation quotients acquire compact genera determined explicitly by the cycle structure of the colored continuation generators. These results unify spectral quantization, analytic continuation, monodromy characters, and covering-space topology within the Lam\'e--Heun correspondence.
\subsection{Applications}

The results have several mathematical-physics applications. In periodic quantum mechanics,
the elliptic double cover and the Floquet--Bloch interpretation connect the Lam\'e spectrum to
finite-gap band structure and transport in periodic media. In cosmological optics, Lam\'e-type
equations occur in distance--redshift and beam-propagation models, so analytic continuation
and connection data provide a global alternative to a single local series representation. In
black-hole and gravitational-wave problems, Heun and confluent-Heun equations arise after
separation of variables; in that setting the singular points represent distinguished boundaries
such as horizons and infinity, while connection coefficients and monodromy data encode global
wave transport. The condition that an unwanted connection coefficient vanish has the same
mathematical form as a resonance or quasinormal-mode condition, and the large-order
coefficient transition derived here suggests an additional numerical diagnostic for locating
spectral values.

The topological results should be interpreted as topology of the complexified differential-
equation domain rather than topology of physical spacetime itself. The monodromy cover is
the natural space on which multivalued solutions become single valued, while the compactified
cover records how branching and singularity data modify the global analytic surface.

\subsection{Open problems and further research}

Several problems remain open. A first problem is to determine the transformation law of the
accessory spectral polynomial $\Delta_N(q;a)$ under the full Lam\'e-preserving anharmonic
subgroup and to identify invariant combinations of its roots. A second is to extend the finite-
sector analysis uniformly to all odd and half-integral Lam\'e couplings, including the required
local exponent gauges. A third is to characterize the Riemann--Hilbert map
$q\mapsto(x(q),y(q),z(q))$ on the Lam\'e Fricke cubic and to identify those accessory values
for which the projective monodromy is reducible, unitary, finite, or algebraic.

A further analytic problem is to develop uniform asymptotic expansions of the Frobenius
coefficients across spectral transitions, including coalescing singularities and confluent Heun
limits. It is also natural to ask whether the coefficient-growth drop from $1$ to $m$ can be
converted into a stable numerical eigenvalue algorithm with rigorous error bounds. On the
topological side, one may classify finite permutation quotients arising from physically relevant
monodromy representations, compute the genera of their compactifications, and study how those
genera change under confluence or parameter degeneration. Finally, a direct application to a
specific astrophysical wave equation---for example a Kerr radial or angular equation---would
allow the accessory parameter, monodromy traces, and connection coefficients developed here
to be compared with observable scattering or resonance data.

\section*{Declaration of competing interest}
The authors declare no known competing financial interests or personal relationships that could have influenced the work reported in this paper.

\section*{Data availability}
No external datasets were used. The numerical values in Fig.~\ref{fig:bands} are generated directly from the finite accessory matrix derived in Section~\ref{sec:qes}.

\end{document}